\documentclass[runningheads]{llncs}
\usepackage{graphicx}
\usepackage{amsfonts,amsmath}
\usepackage{xcolor}
\newtheorem{assumption}{Assumption}

\newcommand{\bd}{\begin{displaymath}}
\newcommand{\ed}{\end{displaymath}}
\newcommand{\be}{\begin{equation}}
\newcommand{\ee}{\end{equation}}
\newcommand{\bea}{\begin{eqnarray}}
\newcommand{\eea}{\end{eqnarray}}
\newcommand{\bda}{\begin{eqnarray*}}
\newcommand{\eda}{\end{eqnarray*}}
\newcommand{\ba}{\begin{array}}
\newcommand{\ea}{\end{array}}

\def\R{\mathbb{R}}
\newcommand{\B}{I\kern -.35em B}

\renewenvironment{proof}[1]{%
\par\vspace{1\baselineskip}%
\noindent{\em Proof#1.\ \ }\ignorespaces }{%
\nobreak\hfill\mbox{\ \ $\Box$}%
\par\vspace{1\baselineskip}%
}

\newcommand{\mt}{\mapsto}
\newcommand{\st}{\subset}

\newcommand{\sm}{\setminus}

\newcommand{\U}{{\cal U}}

\newcommand{\e}{\varepsilon}

\newcommand{\sth}{ \, :\;}

\newcommand{\dd}{\mbox{\rm\,d}}

\newcommand{\Y}{{\cal Y}}
\newcommand{\Z}{{\cal Z}}

\def\R{\mathbb{R}}

\begin{document}
\title{H\"older Regularity in Bang-Bang Type Affine Optimal Control Problems\thanks{Supported by 
the Austrian Science Foundation (FWF) under grant No P31400-N32.}}

%\iffalse
%\\author{First Author\inst{1}\orcidID{0000-1111-2222-3333} \and
%	Second Author\inst{2,3}\orcidID{1111-2222-3333-4444} \and
%	Third Author\inst{3}\orcidID{2222--3333-4444-5555}}
%\institute{Princeton University, Princeton NJ 08544, USA \and
%	Springer Heidelberg, Tiergartenstr. 17, 69121 Heidelberg, Germany
%	\email{lncs@springer.com}\\
%	\url{http://www.springer.com/gp/computer-science/lncs} \and
%	ABC Institute, Rupert-Karls-University Heidelberg, Heidelberg, Germany\\
%	\email{\{abc,lncs\}@uni-heidelberg.de}}
%\fi

\author{Alberto Dom\'inguez Corella\inst{1}\and
Vladimir M. Veliov\inst{1}}
\institute{Institute of Statistics and Mathematical Methods in Economics,
	Vienna University of Technology, Austria }
\maketitle              % typeset the header of the contribution
\begin{abstract}
This paper revisits the issue of H\"older Strong Metric sub-Regularity (HSMs-R) of the optimality system associated with ODE
optimal control problems that are affine with respect to the control. The main contributions are as follows. First, 
the metric in the control space, introduced in this paper, differs from the ones used so far in the literature
in that it allows to take into consideration the bang-bang structure of the optimal control functions. This is especially 
important in the analysis of Model Predictive Control algorithms. Second, the obtained sufficient conditions 
for HSMs-R extend the known ones in a way which makes them applicable to some problems which are 
non-linear in the state variable and the H\"older exponent is smaller than one (that is, the regularity is not Lipschitz).

%We investigate the stability of bang bang solutions of affine control problems due to perturbations. 
%We employ a metric regularity approach introduce a metric that captures de bang-bang properties of a reference 
%solution in the set of controls. The stability with respect to this metric implies stability with respect to other metrics
% induced by norms, extending the previous results for this type of problems.

\end{abstract}

%################################################################################

\section{Introduction}
Consider the following affine optimal control problem 
\begin{align}\label{EOF}
     \min_{u\in\mathcal U} \Big\{l(x(T))+\int_0^T  \Big[w\big(t,x(t)\big) + \langle s\big(t,x(t)\big), u(t)\rangle \Big] \,dt \Big\},
\end{align}
subject to
\begin{align}\label{Ex}
\dot x(t) = a\big(t,x(t)\big)+B\big(t,x(t)\big) u(t) ,  \quad  x(0)=x_0.
\end{align}
Here the state vector $x(t)$ belongs to $\R^n$ and the control function belongs to the set $\mathcal U$ 
of all Lebesgue measurable functions $u:[0,T]\to U$, where $U\subset \R^m$.
Correspondingly, $l:\mathbb R^n\to\mathbb R$ and $w: \R \times \R^n \to \R$ are real-valued functions, 
$s : \R\times \R^n \to \R^m$ and $a:\R \times \R^n \to \R^n$ 
are vector-valued functions,  and $B$ is an $(n \times m)$- matrix-valued function taking values 
in $\R \times \R^n$. The initial state $x_0 \in \R^n$ and the final time $T>0$ are fixed.

We make the following basic assumption.

\begin{assumption}\label{A1}
The set $U$ is a convex compact polyhedron. 
%There exists a convex compact set $X \st \R^n$ such that for every $u \in \U$ equation (\ref{Ex}) has a unique solution $x$
%and $x(t) \in X$ for $t \in [0,T]$.
The functions $f: \R\times \mathbb R^n\times\mathbb R^m\to \mathbb R^n $ and 
$g:\mathbb R\times \mathbb R^n\times\mathbb R^m\to \mathbb R^m $ given by
\begin{align*}
	f(t,x,u) := a(t,x)+B(t,x)u,\quad g(t,x,u) :=w(t,x)+\langle s(t,x), u\rangle,
\end{align*}
and $l:\mathbb R^n\to\mathbb R$ are measurable and bounded in $t$, locally uniformly in $(x,u)$,
and differentiable in $x$. Moreover, these functions  and their first derivatives in $x$ are
Lipschitz continuous in $x$, uniformly in $(t,u) \in [0,T] \times U$.
\end{assumption}

With the usual definition of the Hamiltonian 
\begin{align*}
	H(t,x,u,p) :=g(t,x,u)+\langle p, f(t,x,u)\rangle,
\end{align*}
the local form of the Pontryagin principle for problem (\ref{EOF})-(\ref{Ex}) can be represented 
by the following {\em optimality system} for $x,u$ and an absolutely continuous function 
$p : [0,T] \to \mathbb R^n$ : for almost every $t\in[0,T]$
\begin{eqnarray}
\label{Exm}          0 &=&   \dot x(t) - f(t,x(t),u(t)), \\
\label{Ex0m}        0 &=&           x(0) - x_0, \\
\label{Epm}        0 &=&  \dot p(t)+\nabla_x H(t,x(t),p(t),u(t)), \\
 \label{EpTm}  0 &=&  p(T) - \nabla l\big(x(T)\big),\\
\label{Eum}       0 &\in& \nabla_u H(t,x(t),p(t),u(t))+ N_U(u(t)), % \label{EMP}
\end{eqnarray}
where $N_U(u)$ is the usual normal cone to the convex set $U$ at $u \in \mathbb R^m$.
The optimality system can be recast as  a generalized equation
\be \label{EFm}
             0 \in \varPhi(x,p,u), 
\ee
where the {\em optimality map} $\Phi$ is defined as
\begin{align}
	\label{EFmap}
	\Phi(x,p,u) := \left( \begin{array}{c}
	- \dot x + f(\cdot,x,u) \\
	x(0)-x_0\\
	\dot p + \nabla_{\!\! x} H(\cdot,x,p,u) \\
	p(T)-\nabla l(x(T))\\
	\nabla_{\!u} H(\cdot,x,p)+N_{\U}(u)
	\end{array} \right). 
\end{align}

We remind the general definition of the property of H\"older Strong Metric sub-Regularity (HSMs-R)
of a map, introduced under this name in \cite{D+Rock-04} and appearing earlier in \cite{Dont-95} 
(see the recent paper \cite{Cibulka+Dontchev+Kruger-18} for a comprehensive analysis of this property). 

\begin{definition} \label{Dsmsr}
Let $(\Y, d_\Y)$ and $(\Z, d_\Z)$ be metric spaces. A set-valued map $\varPhi:\Y \to \Z$ 
is strongly H\"older sub-regular at $\hat y\in \mathcal Y$ 
for $\hat z \in Z$ with exponent $\theta > 0$ if $\hat z \in\Phi(\hat y)$ and there exist positive numbers $a,b$ and $\kappa$ 
such that if $y\in\mathcal Y$ and $z\in \mathcal Z$ satisfy
\begin{align*}
	i)\quad z\in\Phi(y)\quad\quad ii)\quad d_{\mathcal Y}(y,\hat y)\le a,\quad\quad iii)\quad d_{\mathcal Z}(z,\hat z)\le b,
\end{align*}
then 
\bd
                 d_\mathcal Y(y,\hat y)\le\kappa d_\mathcal Z(z,\hat z)^\theta.
\ed
We call $a,b$ and $\kappa$  parameters of strong H\"older sub-regularity. 
If $\theta = 1$, then the property is called SMs-R.
\end{definition} 

In this paper, we reconsider this property for the optimality map $\Phi$, with an appropriate definition of the 
metric space where $y = (x,p,u)$ takes values and of the image space. 
It is well known that the HSMs-R property of the optimality map plays an important role in the analysis 
of stability of the solutions and of approximation methods in optimization, in general. We refer to 
\cite{Cibulka+Dontchev+Kruger-18} for general references, and to \cite{P+S+V}, 
where more bibliography on the utilization of the HSMs-R property in the error analysis of optimal control problems is provided. 
We mention that a sufficient condition for SMs-R follows from the fundamental paper \cite{DH-93}, but it does not apply to 
affine problems. 

The paper contains two main contributions. 

(i) Usually in the investigations of regularity of the optimality map for affine problems (see \cite{Os+Ve-20} 
and the bibliography therein)
the metric in the control space is related to the $L^1$-norm, which does not give information about the 
structure of the control function even if the optimal control is of bang-bang type, as assumed later in this paper.
The metric in $\U$ that we define in the present paper captures some structural similarities of the controls,
thus the regularity property in this metric is closer to (but weaker then) the so called {\em structural stability},
investigated in e.g. \cite{Fel2003,Fel+Pog+Stef-09}. The SMs-R or HSMs-R properties of the optimality map $\Phi$
in this metric  is especially important in the analysis of Model Predictive Control algorithms. 

(ii) The obtained sufficient conditions for HSMs-R extend the known ones 
(e.g.  \cite{Alt-2016,Os+Ve-20}) in a way which makes them 
applicable to some problems which are non-linear in the state variable and 
the H\"older exponent $\theta$ is smaller than one.

%###########################################################################
\section{The main result}

First of all we define the metric spaces $\Y$ and $\Z$ of definition and images of the set-valued map $\Phi$
in (\ref{EFm}), (\ref{EFmap}). For that we introduce some notations. 

Using geometric (rather than analytic) terminology, we denote by $V$ the set of
vertices  of $U$, and by $E$ the set of all unit vectors $e \in \R^m$ that are parallel to some edge of $U$. 
Let $Z$ be a fixed non-empty subset of $[0,T]$. For $\e \geq 0$ and for $u_1, u_2 \in \U$ denote
$\Sigma(\e) := [0,T] \sm (Z + [-\e, \e])$, and for $u_1, u_2 \in \U$ define
\bd
    d^*(u_1,u_2):=\inf \left\{ \e>0 \sth u_1(t) = u_2(t)  \text{ for a.e. } t \in \Sigma(\e) \right\}. 
\ed
For $Z = \emptyset$ we formally define $d^*(u_1,u_2) = 0$ if $u_1 = u_2$ a.e., and $d^*(u_1,u_2) = T$ else.
It is easy to check that $d^*$ is a shift-invariant metric in $\U$. For a shift-invariant metric $d$ in any metric space we 
shorten the notation $d(y_1,y_2)$ as $d(y_1-y_2)$.
Then we define the spaces
\bda
           \Y &:=&   W^{1,1}\big([0,T];\R^n\big) \times W^{1,1}\big([0,T]; \R^n\big) \times \U, \\
         \Z &:=& L^1\big([0,T];\R^n\big)\times \R^n \times L^1\big([0,T];\R^n\big) \times\R^n \times L^\infty\big([0,T];\R^n\big)
\eda
with the metrics 
\bd
          d_\Y(x,p,u) := \|x\|_{1,1} + \|p\|_{1,1} + d^*(u), 
\ed
\bd
                 d_\Z(\xi,\eta,\pi,\zeta,\rho) := \|\xi\|_1 + |\eta |+ \|\pi\|_1 + |\zeta| + \|\rho\|_\infty.
\ed
The particular set $Z$ in the definition of $d^*$ will be defined in the next lines. The map $\Phi$ defined in
(\ref{EFmap}) is now considered as a map acting on $\Y$ with images in $\Z$.
The normal cone $N_{\mathcal U}(u)$ to the  closed convex set $\U \st L^1\big([0,T];\R^m\big)$ that appears in (\ref{EFmap}) 
is a subset of the dual space $L^\infty\big([0,T];\R^m\big)$, which can be equivalently defined as
\begin{align*}
	N_{\U}(u) := \begin{cases} \emptyset & \mbox{if } u\notin \U \\
	\{ v \in L^\infty\big([0,T];\R^m\big): \ v(t)  \in N_{U}(u(t))  \mbox{ for a.e. $t \in [\tau,T]$} \} & \mbox{if } u\in \U.
	\end{cases}
\end{align*}

By a standard argument, problem (\ref{EOF})--(\ref{Ex}) has a solution, hence system (\ref{EFm}) has a solution, as well.
Let $\hat y = (\hat x, \hat p, \hat u) \in \Y$ be a reference solution of the optimality system (\ref{EFm}).
Denote by 
\bd
        \hat \sigma := \nabla_{\!u} H(\cdot,\hat x,\hat p) = B(\cdot,\hat x)^\top \hat p+s(\cdot,\hat x)
\ed 
the so-called \textit{switching function} corresponding to the triple $(\hat x,\hat p,\hat u)$. 
We extend the definition of the switching function in the following way. For any $u \in \U$ define the function
\bd
        [0,T] \ni t \mt \sigma[u](t) := B(t,x(t))^\top p(t)+s(t,x(t)),
\ed
where $(x,p)$ solves the system (\ref{Exm})--(\ref{EpTm}) for the given $u$.

\begin{assumption}\label{A2} 
There exists numbers $\gamma_0 > 0$, $\alpha_0>0$ and  $\nu \ge1$ such that 
\bd
         \int_0^T \langle \sigma[u](t), u(t) - \hat u(t) \rangle \dd t \geq \gamma_0 \| u - \hat u \|_1^{\nu+1}
\ed
for all $u \in \U$ with $\|u - \hat u\|_1\le \alpha_0$.
\end{assumption}

This following assumption is standard in the literature on affine optimal control problems, 
see e.g. \cite{CVQ,Fel2003,Os+Ve-20}.

\begin{assumption}\label{A3} 
There exist numbers $\tau > 0$ and $\mu>0$ such that if $s \in[0,T]$ is a zero of $\langle \hat\sigma, e \rangle$ 
for some $e \in E$, then 
\begin{align*}
	|\langle \hat\sigma(t), e\rangle|\ge\mu|t-s|^\nu,
\end{align*}
for all $t\in [s-\tau,s+\tau] \cap [0,T]$. Here $\nu$ is the number from Assumption 2.
\end{assumption}

Assumption 1 implies, in particular, that the set 
\bd
     \hat Z:=\left\{ s \in[0,T] \sth \langle\hat \sigma(s),e \rangle=0 \, \text{ for some }\, e\in E \right\}
\ed
is finite. In what follows the set $Z$ in the definition of the metric $d^*$ will be fixed as $Z = \hat Z$.

The following result is well-known for a box-like set $U$; under the present assumptions it is proved in 
 \cite[Proposition 4.1]{Os+Ve-20}.

\begin{lemma} \label{Lbb}
Under Assumptions 1 and 3, $\hat u$ is (equivalent to) a piecewise constant function with values in the set 
$V$ of vertices of $U$. Moreover, there exists a number $\gamma > 0$ such that 
\begin{align*}
	\int_{0}^T \langle \hat\sigma(t),u(t)-\hat u(t)\rangle\,dt\ge\gamma\|u-\hat u\|_1^{\nu+1}
\end{align*}
for all $u\in\mathcal U$.
\end{lemma}

As a consequence of the lemma, Assumption 2 is implied by Assumption 3, 
provided that there exists $\gamma_1 < \gamma$ such that 
\be \label{EH70}
       \int_{0}^{T} \langle \sigma[\hat u+v](t)-\sigma[\hat u](t), v(t) \rangle \dd t \geq - \gamma_1 \| v \|_1^{\nu + 1}
\ee
for all $v \in \U - \hat u$ with $\| v \|_1\le \alpha_0$. Notice that in the case of a linear-quadratic problem, 
condition (\ref{EH70}) reduces to the one in \cite[Corollary 3.1]{Os+Ve-20}.

The main result in this paper follows.

\begin{theorem}\label{T1}
Let Assumption \ref{A1}--\ref{A3} be fulfilled. There exist positive numbers $a,b$ and $\kappa$ such that if $y=(x,p,u)\in \Y$ and $z\in \Z$  satisfy  
\begin{align*}
i)\quad z\in\Phi(x,p,u)\quad\quad ii)\quad \|u-\hat u\|_1\le a,\quad\quad iii)\quad d_{\mathcal Z}(z,\hat z)\le b,
\end{align*}
then 
\bd
d_\mathcal Y(y,\hat y)\le\kappa d_\mathcal Z(z,\hat z)^{-\nu^2}.
\ed
In particular, the optimality map 
$\Phi:\Y \to \Z$ is  H\"older strongly metrically sub-regular at $\hat y = (\hat x,\hat p,\hat u)$ for zero with exponent $1/\nu^2$.
\end{theorem}

%It is important to stress that the parameters $a$, $b$, and $\kappa$ in Definition \ref{Dsmsr}
%depend on the constants in Assumptions 2 and 3, on the size of the set $U$ and it shortest edge, ...  
%but can be chosen independent of ... !!!!!!!!????????????????

%##################################################################
\section{Proof of the theorem} \label{SProof}

We begin with two lemmas.

\begin{lemma} \label{LH2}
      There exists a  positive number $\kappa_0$ such that for every $\e \in (0,T)$, $t \in \Sigma(\e)$, and $e \in E$
it holds that 
\bd
          | \langle \hat \sigma(t), e \rangle| \geq \kappa_0 \e^\nu.
\ed
\end{lemma}

\begin{proof}{}
	 For brevity we use the notations
\bd
                  \hat \sigma_e(t) :=   \langle \hat \sigma(t), e \rangle, \qquad e \in E,
\ed
\bd
       \delta := \inf \{| \hat \sigma_e(t) | \sth e \in E, \, t \in \Sigma(\tau) \} > 0.
\ed
Let $t \in \Sigma(\tau)$. Then
\bd
        | \hat \sigma_e(t) | \geq \delta = \frac{\delta}{\e^\nu} \e^\nu \geq  \frac{\delta}{T^\nu} \e^\nu.
\ed
Now let $t \in \Sigma(\e) \sm \Sigma(\tau)$. This implies, in particular, that the set $Z = \hat Z$ is non-empty, since
$\Sigma(\tau) = [0,T]$ if $\hat Z = \emptyset$.
Then for some $s \in \hat Z$ and $e \in E$ with $\e \leq | t-s | \leq \tau$ it is fulfilled that
\bd
              | \hat \sigma_e(t) | \geq \mu | t - s |^\nu \geq \mu \e^\nu.
\ed
This implies the claim of the lemma with $\kappa_0 = \min\{ \mu, \delta/T^\nu \}$.
\end{proof}

\begin{lemma} \label{LH3}
There exist  positive numbers $\kappa_1$ and $\rho_1$ such that for every functions $\sigma \in L^\infty$ with $\|\hat\sigma-\sigma\|_\infty\le\rho_0$ and $u \in \U$
	with $\sigma(t) + N_U(u(t)) \ni 0$ for a.e. $t \in [0,T]$ it holds that 
\bd
            u(t) = \hat u(t) \;\mbox{  for a.e. } \, t \in \Sigma\big(\kappa_1 \| \sigma - \hat \sigma \|_\infty^\frac{1}{\nu} \big).
 \ed
\end{lemma} 

\begin{proof}{}
Consider the case $\| \sigma - \hat \sigma \|_\infty > 0$ and $\hat Z\neq0$, the other cases are similar.
For a vertex $v \in V$ we denote 
\bd
             E(v) = \Big\{  \frac{v'-v}{|v'-v|}  \sth v' \mbox{ is a neighboring vertex to } v  \Big\} \st E.
\ed
From (\ref{Eum}) applied to $(\hat x, \hat p, \hat u)$, we obtain that
$\langle \hat \sigma(t), v - \hat u(t) \rangle \geq 0$  for a.e. $t$ and for every $v \in V$.
From Lemma \ref{Lbb} we know that $\hat u(t) \in V$ for a.e. $t$. This implies that for a.e. $t$ it holds that
$\hat \sigma_e(t) \geq 0$ for all $e \in E(\hat u(t))$.
Let us fix such a $t$ which, moreover, belongs to $\Sigma\big(\kappa_1 \| \sigma - \hat \sigma \|_\infty^\frac{1}{\nu} \big)$;
the number $\kappa_1$ will be defined in the next lines.
Then according to Lemma \ref{LH2} we have that
\bd
        \hat \sigma_e(t) \geq \kappa_0 \Big(\kappa_1 \| \sigma - \hat \sigma \|_\infty^\frac{1}{\nu}\Big)^\nu = 
               \kappa_0 (\kappa_1)^\nu \| \sigma - \hat \sigma \|_\infty.
\ed
Let us choose $\kappa_1$ and $\rho_1$ such that $\kappa_0\kappa_1^\nu > 1$ and $\rho_1<(T/\kappa_1)^\nu$. Then 
\bd
          \sigma_e(t) := \langle \sigma(t), e \rangle = \hat \sigma_e(t) + (\sigma_e(t) - \hat \sigma_e(t))
              >  \| \sigma - \hat \sigma \|_\infty -  \| \sigma - \hat \sigma \|_\infty = 0.
\ed    
Thus we obtain that 
\bd
                  \langle \sigma(t), v - \hat u(t) \rangle > 0 \;\; \mbox{ for every } v \in V \sm \{\hat u(t)\}.
\ed
This implies that $\hat u(t)$ is the unique solution of $\sigma(t) + N_U(u) \ni 0$, hence $u(t) = \hat u(t)$.
\end{proof}

\begin{proof}{ of Theorem \ref{T1}}
In the proof we use the constants involved in the assumptions and in the lemmas above. 	Let $z=(\xi,\eta,\pi,\upsilon,\rho) \in \Z$ and $y=(\tilde x, \tilde \lambda,\tilde u) \in \Y$ 
such that $z \in \Phi(y)$.
Denote $\Delta:=\rho+ [\sigma_\tau[\tilde u] - \nabla_u H(\bar p, \tilde x, \tilde \lambda)]$.
Using the Gr{\"o}nwall's inequality, we can find constants $c_1$ and $c_2$ (indepent of $y$ and $z$) 
such that $\|\Delta\|_{\infty}\le c_1\|z\|_{\Z}$ and $\|y-\hat y\|_{\Y}\le c_2 \| \tilde u-\hat u \|_1$. Let $a:=\alpha_0$, since $\Delta - \sigma[\tilde u] = \rho - \nabla_u H(\bar p,\tilde x,\tilde \lambda) \in N_{\U}(\tilde u)$,
we have 
\begin{align*}
	\int_{0}^{T}\langle \Delta-\sigma[\tilde u],\hat u-\tilde u\rangle\le0.
\end{align*}
Now, by Assumption \ref{A2}
\begin{align*}
	0&\ge\int_{0}^{T}\langle \Delta-\sigma[\tilde u],\hat u-\tilde  u \rangle
	=\int_{0}^{T}\langle \sigma[\tilde u],\tilde u-\hat u\rangle+
	\int_{0}^{T}\langle \Delta,\hat u-\tilde  u\rangle\\
	&\ge \gamma_0\left( \int_{0}^{T}|\tilde u-\hat u|\right)^{\nu+1}-\|\Delta\|_{\infty}\int_{0}^{T}|\tilde u-\hat u|.
\end{align*}
Hence, 
\bd
\|\tilde u-\hat u\|_1\le\frac{1}{\gamma_0^\frac{1}{\nu}}\|\Delta\|_{\infty}^{\frac{1}{\nu}}\le\frac{c_1^{\frac{1}{\nu}}}{\gamma_0^{\frac{1}{\nu}}}\|z\|_{\mathcal Z_\tau}^{\frac{1}{\nu}}.
\ed
With $\kappa':=c_2 c_1^\frac{1}{\nu} \gamma_0^{-\frac{1}{\nu}}$ we obtain that
\begin{align*}
	\|y-\hat y\|_{\mathcal Y}\le\kappa'\|z\|_{\mathcal Z}^\frac{1}{\nu}.
\end{align*}
There exists a constant $c_3>0$ such that $\|\sigma[\tilde u]-\sigma[\hat u]\|_\infty\le c_3\|\tilde u-\hat u\|_1$, hence 
\begin{align*}
	\|\sigma[\tilde u]-\rho-\sigma[\hat u]\|_\infty\le c_3\kappa'\|z\|_{\mathcal Z}^\frac{1}{\nu}+\|\rho\|_{\infty}\le(c_3\kappa'+1)\|z\|_{\mathcal Z}^\frac{1}{\nu}.
\end{align*}
Let $b$ small enough so Lemma 3 holds with $\sigma=\sigma[\tilde u]-\rho$. We get $d^*(\tilde u,\hat u)\le\left[ (c_3\kappa'+1)\|z\|_{\mathcal Z}^\frac{1}{\nu}\right] ^{\frac{1}{\nu}}=(c_3\kappa'+1)^\frac{1}{\nu} \|z\|_{\mathcal Z}^\frac{1}{\nu^2}$. Finally,
\begin{align*}
d_{\mathcal Y}(y,\hat y)\le (c_3\kappa'+1)^\frac{1}{\nu} \|z\|_{\mathcal Z}^\frac{1}{\nu^2}+k'\|z\|_{\mathcal Z}^\frac{1}{\nu}\le \kappa\|z\|_{\mathcal Z}^\frac{1}{\nu^2},
\end{align*}
where $\kappa:=(c_3\kappa'+1)+\kappa'$.
\end{proof}

%####################################################################
\section{An example} \label{SEx}

Let $\alpha: [0, +\infty) \to \R$ be a differentiable function with Lipschitz derivative 
such that $\alpha$ attains its minimum at zero and  $\alpha'(x)\ge0$ for all 
$x\in [0, +\infty)$, e.g., $\alpha(x)= x^2$ or $\alpha(x)=1-e^{-x^2}$. Moreover, let $T$ be a positive number and let 
$\beta :[0,T] \to \R$ be $\nu$-times differentiable ($\nu \geq 1$) and satisfy $\beta(t) > 0$ for $t >0$, 
$\beta(0) = \ldots = \beta^{(\nu-1)}(0) = 0$, $\beta^{(\nu)}(0) \not= 0$.
Consider the following optimal control problem

\begin{align}\label{ex1}
\min_{u\in\mathcal U} \Big\{\int_0^T  \Big[\alpha\big(x(t)\big) + \beta(t) u(t) \Big] \dd t \Big\},
\end{align}
subject to
\begin{align}
\dot x(t) = u(t) ,  \quad  x(0)=0,\label{ex2}\quad u(t)\in[0,1]\quad \text{a.e. in $[0,T]$}.
\end{align}
The optimality system of problem (\ref{ex1})-(\ref{ex2}) is given by
\begin{eqnarray}
\label{Ex1}          0 &=&   \dot x(t) - u(t), \quad x(0)=0,\\
\label{Ex3}        0 &=&  \dot p(t)+\alpha'\big(x(t)\big),\quad  p(T)=0, \\
\label{Ex5}       0 &\in& \beta(t) +p(t)+ N_{[0,1]}(u(t)). % \label{EMP}
\end{eqnarray}
Hence the switching function corresponding to each control $u\in\mathcal U$ is given by 
$\sigma[u](t):= \beta(t) +p[u](t)$, where $p[u](t)=\int_{t}^{T}\alpha'\big(x[u](s)\big)\, \dd s$ and 
$x[u](t)=\int_{0}^{t} u(s) \dd s$, $t\in[0,T]$. It is clear that the unique minimizer of problem (\ref{ex1})-(\ref{ex2}) is 
$(\hat x,\hat p,\hat u)=(0,0,0)$, and consequently its switching function is given by $\hat{\sigma}(t)= \beta(t)$. 
Since $\hat{\sigma}$ satisfies $\hat\sigma(0) = \ldots = \hat\sigma^{(\nu-1)}(0) = 0$ and $\hat\sigma^{(\nu)}(0)\neq 0$, 
we have that Assumption \ref{A3} 
is satisfied with the same number $\nu$ but not with $\nu-1$. Now, observe that for $v\in\mathcal U-\hat u$ we have
\begin{align*}
 \int_{0}^{T} \langle \sigma[\hat u+v](t)-\sigma[\hat u](t), v(t) \rangle \dd t=\int_{0}^{T} \langle p[\hat u+v](t), v(t) \rangle \dd t \ge 0.
\end{align*}
Thus, in accordance with (\ref{EH70}), Assumption \ref{A2} is satisfied.
%##########################################################

%\bibliography{ref1}{}
%\bibliographystyle{splncs04}
\end{document}